\documentclass[10pt]{article}
\usepackage[pagebackref,colorlinks=true]{hyperref}
\usepackage{amssymb, amsfonts, amsthm, mathrsfs}
\usepackage{graphicx, xcolor}
\usepackage[all]{xy}
\usepackage[leqno]{amsmath}

\theoremstyle{definition}
\newtheorem{thm}{Theorem}[section]
\newtheorem{lemma}[thm]{Lemma}
\newtheorem{cor}[thm]{Corollary}
\newtheorem{prop}[thm]{Proposition}

\newtheorem{rmk}[thm]{Remark}

\makeatletter

\def\id{\mathrm{id}}

\def\tr{\mathrm{Tr}}
\def\ccc{\mathbb{C}}

\def\hh{\mathbb{H}}
\def\zz{\mathbb{Z}}
\def\rr{\mathbb{R}}
\def\pp{\mathbb{P}}

\def\clo{\mathcal{O}}
\def\cls{\mathcal{S}}

\def\frI{\mathfrak{I}}
\def\pt{\partial}
\def\bpt{\overline{\pt}}

\def\ud{\mathrm{d}}

\def\bz{\overline{z}}
\def\bu{\overline{u}}
\def\bw{\overline{w}}

\def\bj{\overline{\jmath}}
\def\bk{\overline{k}}

\def\bzeta{\overline{\zeta}}

\def\bon{\bar{1}}
\def\btw{\bar{2}}

\def\vol{\mathrm{vol}}

\makeatother

\title{Some Torsional Local Models of Heterotic Strings}
\author{Teng Fei}

\begin{document}
\maketitle{}

\section{Introduction}

In his study of heterotic superstrings with torsion, Strominger \cite{strominger1986} derived a system of partial differential equations from considerations of supersymmetry and anomaly cancellation, which is later known as the \emph{Strominger system}. In real dimension 6, the internal manifold $X$ has to be a complex 3-fold with trivial canonical bundle. Moreover, $X$ is equipped with an Hermitian metric $\omega$ and an Hermitian holomorphic vector bundle $(E,h)$. Let $\Omega$ be a nowhere vanishing holomorphic (3,0)-form on $X$, then the Strominger system consists of the following equations\footnote{Strominger's original paper \cite{strominger1986} uses a different normalization for Equation (\ref{ac}).}:
\begin{eqnarray}
\label{hym}F\wedge\omega^2=0,\quad F^{0,2}=F^{2,0}=0,\\
\label{ac}\sqrt{-1}\pt\bpt\omega=\frac{\alpha'}{4}(\tr(R\wedge R)-\tr(F\wedge F)),\\
\label{cb'}\ud^*\omega=i(\pt-\bpt)\log\|\Omega\|_\omega.
\end{eqnarray}

In the above equations, $\alpha'$ is a positive coupling constant, $R$ and $F$ are curvature 2-forms of $T^{1,0}X$ and $E$ respectively, computed with respect to certain metric connections. As the Green-Schwarz anomaly cancellation mechanism works for arbitrary connection \cite{hull1986}, we have the freedom to choose the Chern connection to compute $R$, as suggested in \cite[Section 4.3]{strominger1986}. We will call Equations (\ref{hym}) and (\ref{ac}) the \emph{Hermitian-Yang-Mills} equation and \emph{anomaly cancellation} equation respectively.

It is very clear that the Strominger system generalizes the complex Monge-Amp\`ere equation used in the torsion-free theory in which $X$ is a K\"ahler Calabi-Yau. Assuming $X$ is K\"ahler, if we further let $E$ be the holomorphic tangent bundle $T^{1,0}X$, then $R=F$ and Equations (\ref{hym})-(\ref{cb'}) are solved automatically. These solutions will be called the \emph{K\"ahler solutions}.

The first attempt to the Strominger system was made in \cite[Section 4.2 \& 4.3]{strominger1986}, where orbifold solutions and perturbative solutions (up to first order) of K\"ahler solutions were obtained.

A huge progress was achieved by Li-Yau \cite{li2005} almost 20 years later. Their contribution has two aspects. Firstly, Li and Yau realized that Equation (\ref{cb'}) is equivalent to
\begin{equation}
\label{cb}\ud(\|\Omega\|_\omega\cdot\omega^2)=0,
\end{equation}
whose geometric meaning is very clear. Recall from \cite{michelsohn1982} that an Hermitian metric $\omega$ is called \emph{balanced} if it is co-closed. As Equation (\ref{cb}) says that $\omega$ is a balanced metric after a conformal change, we will call this equation the \emph{conformally balanced equation}. Such an interpretation imposes a mild topological restriction on $X$, on which we refer to Michelsohn's intrinsic characterization of balanced manifolds \cite[Theorem 4.7]{michelsohn1982}. Secondly, Li and Yau were able to obtain the first irreducible smooth solutions to the Strominger system by deforming the K\"ahler solutions. Such techniques were further developed in \cite{andreas2012, andreas2012b} to deal with more general perturbations.

Another breakthrough was due to Fu, Yau and their collaborators. They observed that on the geometric models constructed by Goldstein-Prokushkin \cite{goldstein2004} (sometimes known as the FSY geometries), a clever choice of ansatz reduces the whole Strominger system to a Monge-Amp\`ere type equation of a single dilaton function on the Calabi-Yau 2-fold base. Using this idea, they were able to get the first genuine non-K\"ahler solutions, on both compact backgrounds \cite{becker2006, fu2007, fu2008} and on local models \cite{fu2009}. Such a method can be further modified to yield more heterotic non-K\"ahler geometries \cite{becker2009}.

It should be mentioned that solutions to the Strominger system have also been found on blow-up of conifold \cite{carlevaro2010}, on various nilmanifolds \cite{fernandez2009, grantcharov2011, fernandez2014, ugarte2014, ugarte2015}, on (quotients of) $SL(2,\ccc)$ \cite{fei2015} and on Calabi-Gray models \cite{fei2015c}.

As a generalization of the complex Monge-Amp\`ere equation solved by Yau \cite{yau1978}, the Strominger system provides a natural framework to find canonical metrics on non-K\"ahler Calabi-Yau manifolds. This is closely related to the moduli problem of Calabi-Yau manifolds via Reid's fantasy \cite{reid1987}. Recall that by blowing down disjoint (-1,-1)-curves in a projective Calabi-Yau 3-fold, one gets a singular Calabi-Yau 3-fold with finitely many ordinary double points. These singularities can be further smoothed out to yield various topologically distinct 3-folds with trivial canonical bundle which are in general non-K\"ahler. Such a geometric manipulation is known as the \emph{conifold transition}. Roughly speaking, Reid's fantasy says that all reasonably nice Calabi-Yau 3-folds can be connected via a sequence of conifold transitions.

The local model of conifold transition was studied by Candelas-de la Ossa \cite{candelas1990} in great details, where explicit Ricci-flat K\"ahler metrics on both deformed and resolved conifold were constructed and their asymptotic behaviors were carefully analyzed. Because conifold transitions in general take place in the non-K\"ahler category, it is of great importance to understand the local model from a non-K\"ahler point of view. That is, to look for non-K\"ahler solutions to the Strominger system on both deformed and resolved conifold.

The first half was partially achieved in \cite{fei2015}, where homogeneous solutions to the Strominger system on deformed conifold were found through an identification with $SL(2,\ccc)$. The second half will be covered in this paper. Indeed, we are able to solve the Strominger system on the resolved conifold, and actually, on much more general local models. The ansatz we use comes from the twistor space construction, which is based on the ideas used in \cite{fu2008, fei2015c, fei2015b}. It should be noticed that our solution has a nontrivial $\tr(F\wedge F)$ term in order to cancel the anomaly, which does not follow from simple topological consideration.

This paper is organized as follows. In Section \ref{s2} we give a brief account of preliminaries, including differential geometry of resolved conifolds, hyperk\"ahler 4-manifolds and their twistor spaces. Section \ref{s3} is devoted to solving the Strominger system on a class of non-compact Calabi-Yau 3-folds which can be obtained by removing a fiber from the twistor fibration of hyperk\"ahler 4-manifolds. Some special examples of these solutions are analyzed in Section \ref{s4}. In particular, we obtain nontrivial solutions on $\ccc^3$ and on resolved conifold. Section \ref{s5} is of somewhat independent interest. We will make a discussion about Chern-Ricci flat balanced metrics on non-K\"ahler Calabi-Yau manifolds and give some examples. This can be viewed as a generalization of Calabi's work \cite{calabi1979}. In addition, we derive the Euler-Lagrange equation for extremal balanced metrics and show that Chern-Ricci flat balanced metrics are extremal (and in certain cases the only extremal ones).

\subsection*{Acknowledgement}

The author would like to thank Prof. Shing-Tung Yau and Prof. Victor Guillemin for their constant encouragement and help. Communications with Bao-Sen Wu are extremely inspiring and fruitful. The author is indebted to various helpful discussions with Peng Gao, Peter Smillie and Cheng-Long Yu as well. Thanks also goes to Prof. Valentino Tosatti for his useful comments.

\section{Preliminaries}\label{s2}

\subsection{Resolved Conifold}

Let $Q$ be the \emph{conifold}, or in other words an affine quadric cone in $\ccc^4$. That is, \[Q=\{(z_1,z_2,z_3,z_4)\in\ccc^4:z_1^2+z_2^2+z_3^2+z_4^2=0\}.\] It is quite obvious that $Q$ has a unique singular point at the origin, which is known as an ordinary double point. The small resolution of ordinary double points was first discovered by Atiyah \cite{atiyah1958} that we shall manifest.

By a linear transformation of coordinates \[\begin{cases}w_1=z_1+iz_2\\ w_2=z_3+iz_4\\ w_3=iz_4-z_3\\ w_4=z_1-iz_2\end{cases},\] we can identify $Q$ as the zero locus of $w_1w_4-w_2w_3$, or more suggestively, \[\det\begin{pmatrix}w_1 & w_2\\ w_3 & w_4\end{pmatrix}=0.\] Now let $\ccc\pp^1$ be parameterized by $\lambda=[\lambda_1:\lambda_2]$. Consider \[\widetilde{Q}=\left\{(w,\lambda)\in\ccc^4\times\ccc\pp^1:\begin{pmatrix}
w_1 & w_2\\ w_3 & w_4\end{pmatrix}\begin{pmatrix}\lambda_1\\ \lambda_2\end{pmatrix}=\begin{pmatrix}0\\ 0\end{pmatrix}\right\}.\] It is not hard to see that $\widetilde{Q}$ is smooth and the first projection \[p_1:\widetilde{Q}\to Q\] is an isomorphism away from $\{0\}\times\ccc\pp^1\subset\widetilde{Q}$. Therefore we call $\widetilde{Q}$ the \emph{resolved conifold} and $p_1:\widetilde{Q}\to Q$ the small resolution of $Q$, as the exceptional set $\ccc\pp^1$ is of codimension 2. Moreover, the second projection $p_2:\widetilde{Q}\to\ccc\pp^1$ allows us to identify $\widetilde{Q}$ with the total space of $\clo(-1,-1)\to\ccc\pp^1$. As a consequence, the resolved conifold $\widetilde{Q}$ has trivial canonical bundle.

\subsection{Hyperk\"ahler 4-manifolds and Their Twistor Spaces}

Let $(N,g)$ be a Riemannian 4-manifold. If in addition $M$ admits three integrable complex structures $I$, $J$ and $K$ with $IJK=-\id$ such that $g$ is a K\"ahler metric with respect to any of $\{I,J,K\}$, then we call $(N,g,I,J,K)$ a hyperk\"ahler 4-manifold.

Now we assume that $N$ is a hyperk\"ahler 4-manifold and denote by $\omega_I$, $\omega_J$ and $\omega_K$ the associated K\"ahler forms with respect to corresponding complex structures. One can easily check that $\omega_J+i\omega_K$ is a holomorphic nowhere vanishing $(2,0)$-form with respect to the complex structure $I$, therefore $(N,I)$ has trivial canonical bundle.

If $N$ is compact, then by the Enriques-Kodaira classification of complex surfaces, $N$ must be either a complex torus or a K3 surface. However, if we allow $N$ to be noncompact, there are many more possibilities. An extremely important class of them is the so-called ALE (asymptotically locally Euclidean) spaces. These spaces were first discovered as gravitational (multi-)instantons by physicists \cite{eguchi1978, gibbons1978} and finally classified completely by Kronheimer \cite{kronheimer1989a, kronheimer1989b}.

It is well-known that a hyperk\"ahler 4-manifold $N$ is anti-self-dual, therefore its twistor space $Z$ is a complex 3-fold \cite{atiyah1978}. Following \cite{hitchin1987}, the twistor space $Z$ can be described as follows. Let $\zeta$ parameterizes $\ccc\pp^1$. We shall identify $\ccc\pp^1$ with $S^2=\{(\alpha,\beta,\gamma)\in\rr^3:\alpha^2+\beta^2+\gamma^2=1\}$ via \[(\alpha,\beta,\gamma)= \left(\frac{1-|\zeta|^2}{1+|\zeta|^2},\frac{\zeta+\bzeta}{1+|\zeta|^2},\frac{i(\bzeta-\zeta)}{1+|\zeta|^2}\right).\] The twistor space $Z$ of $N$ is defined to be the manifold $Z=\ccc\pp^1\times N$ with the almost complex structure $\frI$ given by \[\frI=j\oplus(\alpha I_x+\beta J_x+\gamma K_x)\]
at point $(\zeta,x)\in\ccc\pp^1\times N$, where $j$ is the standard complex structure on $\ccc\pp^1$ with holomorphic coordinate $\zeta$. It is a theorem of \cite{hitchin1987} that $\frI$ is integrable and the projection $p:Z\to\ccc\pp^1$ is a holomorphic fibration. We shall call $p:Z\to\ccc\pp^1$ the twistor fibration.

The twistor spaces of ALE spaces can be described in other ways. For instance, the twistor spaces of ALE spaces of type $A$ can be constructed very concretely using algebraic geometry \cite{hitchin1979}. For later use, we shall give a different characterization of the $A_1$ case. i.e. the twistor space of Eguchi-Hansen space, as Hitchin did in \cite[Section 8]{hitchin1981}.

Let $Q$ and $\widetilde{Q}$ be conifold and resolved conifold described above. Consider the map  \[\rho=z_4\circ p_1:\widetilde{Q}\xrightarrow{p_1} Q\xrightarrow{z_4}\ccc.\] It is obvious that, when $z_4\neq0$, the fiber $\rho^{-1}(z_4)$ is a smooth affine quadric in $\ccc^3$. After a little work, we can see that $\rho^{-1}(0)$ is biholomorphically equivalent to the total space of $T^*\ccc\pp^1$. It follows that $\rho$ is a fibration.

Now let $\rho':\widetilde{Q}'\to\ccc$ be another copy of $\rho:\widetilde{Q}\to\ccc$. We may glue these two fibrations holomorphically by identifying $\rho^{-1}(\ccc^\times)\xrightarrow{\rho}\ccc^\times$ with $\rho'^{-1}(\ccc^\times)\xrightarrow{\rho'}\ccc^\times$ via \[(z'_1,z'_2,z'_3,z'_4)=\left(\dfrac{z_1}{z_4^2},\dfrac{z_2}{z_4^2},\dfrac{z_3}{z_4^2},\dfrac{1}{z_4}\right).\] As a consequence, we get a holomorphic fibration over $\ccc\pp^1$, which is exactly the twistor fibration of Eguchi-Hansen space.

We conclude that, when performing hyperk\"ahler rotations, there are exactly two complex structures on the Eguchi-Hansen space up to biholomorphism. There is a pair of two antipodal points on $\ccc\pp^1$, over which the fibers of twistor fibration are biholomorphic to the total space of $T^*\ccc\pp^1$. We shall call these fibers \emph{special}. All the other fibers are biholomorphic to the smooth affine quadric in $\ccc^3$. A key observation from this construction is the following proposition.
\begin{prop}(Hitchin, \cite[Section 8]{hitchin1981})\label{resolved}
If we remove a special fiber from the total space of the twistor fibration of Eguchi-Hansen space, then what is left is biholomorphic to the resolved conifold.
\end{prop}

\section{Local Torsional Models of Heterotic Strings}\label{s3}

Let us first briefly recall the results in \cite{fei2015c}.

Given any hyperk\"ahler 4-manifold $N$ and any oriented minimal surface $\Sigma_g$ of genus $g$ in flat $T^3$, we can cook up a complex structure on $X=\Sigma_g\times N$ such that $X$ is a non-K\"ahler Calabi-Yau 3-fold. Indeed, $X$ can also be obtained as the total space of the pullback of the twistor fibration of $N$ through the Gauss map of $\Sigma_g$. In the case of $N=T^4$, by making use of the fibration structure, we were able to solve the whole Strominger system on $X$. However, the solution metric is degenerate on the fibers over the branched points of the Gauss map.

This fact suggests that it is probably more natural to consider the Strominger system on the twistor space $Z$. However, the twistor space $Z$ can never have trivial canonical bundle hence the Strominger system does not make sense. In order to remedy this problem, we need to remove some parts of $Z$ to get a noncompact complex 3-fold $Y\subset Z$ with trivial canonical bundle. It turns out that we only need to remove one single fiber of the twistor fibration.

Now let $N$ be an arbitrary hyperk\"ahler 4-manifold and $Z$ its twistor space with twistor fibration $p:Z\to\ccc\pp^1$. Without loss of generality, let $F_\infty$ be the fiber of $p:Z\to\ccc\pp^1$ over $\zeta=\infty\in\ccc\pp^1$ and we may remove $F_\infty$ from $Z$ to get a noncompact space $Y:=Z\setminus F_\infty$. One can check directly that $Y$ has trivial canonical bundle which can be trivialized by \[\Omega:=(-2\zeta\omega_I+(1-\zeta^2)\omega_J+i(1+\zeta^2)\omega_K)\wedge\ud \zeta.\] The main theme of this section is to write down a solution to the Strominger system on $Y$. As we shall see, the recipe for the ansatz is almost identical to the one we used in \cite{fei2015c}. However the curvature of $N$ will come into play and our calculation will be much more complicated. Moreover, we have to take the Hermitian-Yang-Mills equation (\ref{hym}) into account as well.

\subsection{The Conformally Balanced Equation (\ref{cb})}

Observe that $Y$ is diffeomorphically a product $\ccc\times N$ with twisted complex structure. Let $h:N\to\rr$ and $g:\ccc\to\rr$ be arbitrary smooth functions. In addition, we use \[\omega_{\ccc\pp^1}=\frac{2i}{(1+|\zeta|^2)^2}\ud\zeta\wedge\ud\bzeta\] to denote the round metric of radius 1 on $\ccc\pp^1$ and its restriction on $\ccc=\ccc\pp^1\setminus\{\infty\}$.

Now consider the Hermitian metric
\begin{equation}\label{ansatz}
\omega=\frac{e^{2h+g}}{(1+|\zeta|^2)^2}(\alpha\omega_I+\beta\omega_J+\gamma\omega_K)+e^{2g}\omega_{\ccc\pp^1}
\end{equation}
on $Y=\ccc\times N$. One can check that
\begin{equation}\label{mod}
\|\Omega\|_\omega=c\cdot\frac{(1+|\zeta|^2)^4}{e^{2h+2g}}
\end{equation}
for some positive constant $c$ and
\begin{equation}\label{bal}
\omega^2=2\frac{e^{4h+2g}}{(1+|\zeta|^2)^4}\vol_N +2\frac{e^{2h+3g}}{(1+|\zeta|^2)^2}(\alpha\omega_I+\beta\omega_J+\gamma\omega_K)\wedge\omega_{\ccc\pp^1},
\end{equation}
where $\vol_N$ is the volume form on $N$. It follows that $\omega$ solves the conformally balanced equation (\ref{cb}) for arbitrary $g$ and $h$ by direct computation.

\subsection{Curvature Computation and the Anomaly Cancellation (\ref{ac})}

Now we proceed to solve the anomaly cancellation equation (\ref{ac}) using ansatz (\ref{ansatz}). The first step would be to compute the curvature term $\tr(R\wedge R)$, using the Chern connection, with respect to the metric (\ref{ansatz}). To do so, following the method used in \cite{fei2015c}, we had better first solve for a local holomorphic frame of $(1,0)$-forms on $Y$.

We fix $I$ to be the background complex structure on the hyperk\"ahler 4-manifold $N$. Locally, there exist holomorphic coordinates $\{z_1,z_2\}$ such that the holomorphic $(2,0)$-form can be expressed as \[\omega_J+i\omega_K=\ud z_1\wedge\ud z_2.\] Let $\kappa$ be the K\"ahler potential, i.e., we have \[\omega_I=i\pt\bpt\kappa=i(\kappa_{1\bar{1}}\ud z_1\wedge\ud\bz_1+\kappa_{1\bar{2}}\ud z_1\wedge\ud\bz_2+\kappa_{2\bar{1}}\ud z_2\wedge\ud\bz_1+\kappa_{2\bar{2}}\ud z_2\wedge\ud\bz_2),\] where lower indices represent partial derivatives with respect to holomorphic or anti-holomorphic coordinates. As $N$ is hyperk\"ahler, we must have \[\kappa_{1\bar{1}}\kappa_{2\bar{2}}-\kappa_{1\bar{2}}\kappa_{2\bar{1}}=\frac{1}{4}.\] It is convenient to list the following equations

\begin{minipage}{0.48\textwidth}
\begin{align*}
I\ud z_1&=i\ud z_1,\quad I\ud z_2=i\ud z_2,\\
J\ud z_1&=-2(\kappa_{2\bar{1}}\ud\bz_1+\kappa_{2\btw}\ud\bz_2), \\
J\ud z_2&=2(\kappa_{1\bon}\ud\bz_1+\kappa_{1\btw}\ud\bz_2),  \\
K\ud z_1&=-2i(\kappa_{2\bon}\ud\bz_1+\kappa_{2\btw}\ud\bz_2), \\
K\ud z_2&=2i(\kappa_{1\bon}\ud\bz_1+\kappa_{1\btw}\ud\bz_2),
\end{align*}
\end{minipage}
~
\begin{minipage}{0.48\textwidth}
\begin{align*}
I\ud\bz_1&=-i\ud\bz_1,\quad I\ud\bz_2=-i\ud\bz_2,\\
J\ud\bz_1&=-2(\kappa_{1\btw}\ud z_1+\kappa_{2\btw}\ud z_2),\\
J\ud\bz_2&=2(\kappa_{1\bon}\ud z_1+\kappa_{2\bon}\ud z_2),\\
K\ud\bz_1&=2i(\kappa_{1\btw}\ud z_1+\kappa_{2\btw}\ud z_2),\\
K\ud\bz_2&=-2i(\kappa_{1\bon}\ud z_1+\kappa_{2\bon}\ud z_2),
\end{align*}
\end{minipage}
\vspace{0.08in}
~\\
which tell us how $I$, $J$ and $K$ acts on 1-forms.

As $Y=\ccc\times N$, we can think of $z_i$ as functions locally defined on $Y$, though they are no longer holomorphic. In the same manner, we shall record how the complex structure $\frI$ on $Y$ acts on 1-forms

\begin{minipage}{0.48\textwidth}
\[\begin{split}\frI\ud\zeta&=i\ud\zeta,\\ \frI\ud z_1&=i\alpha\ud z_1-2(\beta+i\gamma)(\kappa_{2\bon}\ud\bz_1+\kappa_{2\btw}\ud\bz_2), \\ \frI\ud z_2&=i\alpha\ud z_2+2(\beta+i\gamma)(\kappa_{1\bon}\ud\bz_1+\kappa_{1\btw}\ud\bz_2),\end{split}\]
\end{minipage}
~
\begin{minipage}{0.48\textwidth}
\[\begin{split}\frI\ud\bzeta&=-i\ud\bzeta,\\ \frI\ud\bz_1&=-i\alpha\ud\bz_1-2(\beta-i\gamma)(\kappa_{1\btw}\ud z_1+\kappa_{2\btw}\ud z_2),\\ \frI\ud\bz_2&=-i\alpha\ud\bz_2+2(\beta-i\gamma)(\kappa_{1\bon}\ud z_1+\kappa_{2\bon}\ud z_2).\end{split}\]
\end{minipage}
\vspace{0.08in}
~\\
Now consider a 1-form \[\theta=L\ud\zeta+A\ud z_1+B\ud z_2+C\ud\bz_1+D\ud\bz_2.\] If $\theta$ is of type (1,0), i.e., $\frI\theta=i\theta$, then we have \[\begin{split}A&=\frac{2i\kappa_{1\btw}}{\zeta}C-\frac{2i\kappa_{1\bon}}{\zeta}D,\\ B&=\frac{2i\kappa_{2\btw}}{\zeta}C-\frac{2i\kappa_{2\bon}}{\zeta}D.\end{split}\]

Let \[\theta_1=2i\frac{\kappa_{1\btw}}{\zeta}\ud z_1+2i\frac{\kappa_{2\btw}}{\zeta}\ud z_2+\ud\bz_1,\quad  \theta_2=2i\frac{\kappa_{1\bon}}{\zeta}\ud z_1+2i\frac{\kappa_{2\bon}}{\zeta}\ud z_2-\ud\bz_2.\] It is easy to see that they are (1,0)-forms. Moreover, we have \[\theta=L\ud\zeta+C\theta_1-D\theta_2.\]

\begin{lemma}
$\theta$ is a holomorphic $(1,0)$-form if and only if
\begin{equation}\label{loc}
2\zeta\bpt L=-C((1-\alpha)\theta_1-2\ud\bz_1)+D((1-\alpha)\theta_2+2\ud\bz_2)
\end{equation}
holds, where $C$ and $D$ satisfy
\begin{equation}\label{simp}
\begin{split}\bpt C&=i(\beta-i\gamma)[C(\kappa_{1\bon\btw}\overline{\theta}_1-\kappa_{2\bon\btw}\overline{\theta}_2) -D(\kappa_{1\bon\bon}\overline{\theta}_1-\kappa_{2\bon\bon}\overline{\theta}_2)]\\
\bpt D&=i(\beta-i\gamma)[C(\kappa_{1\btw\btw}\overline{\theta}_1-\kappa_{2\btw\btw}\overline{\theta}_2) -D(\kappa_{1\bon\btw}\overline{\theta}_1 -\kappa_{2\bon\btw}\overline{\theta}_2)].\end{split}
\end{equation}
\end{lemma}
\begin{proof}
The proof follows from straightforward calculation. It is quite clear that $\theta$ is holomorphic if and only if \[\bpt L\wedge\ud t+\bpt C\wedge\theta_1-\bpt D\wedge\theta_2+C\bpt\theta_1-D\bpt\theta_2=0.\]

One can compute directly that
\[\begin{split}\pt\theta_1&=-\frac{1+\alpha}{2\zeta}\ud\zeta\wedge\theta_1,\\ \bpt\theta_1&=-\frac{\ud \zeta}{\zeta}\wedge\left(\frac{1-\alpha}{2}\theta_1-\ud\bz_1\right)-\frac{2i}{\zeta}(\kappa_{1\bon\btw}\ud z_1\wedge\ud\bz_1+\kappa_{1\btw\btw}\ud z_1\wedge\ud\bz_2+\kappa_{2\bon\btw}\ud z_2\wedge\ud\bz_1+\kappa_{2\btw\btw}\ud z_2\wedge\ud\bz_2),\\ \pt\theta_2&=-\frac{1+\alpha}{2\zeta}\ud\zeta\wedge\theta_2,\\ \bpt\theta_2&=-\frac{\ud \zeta}{\zeta}\wedge\left(\frac{1-\alpha}{2}\theta_2+\ud\bz_2\right)-\frac{2i}{\zeta}(\kappa_{1\bon\bon}\ud z_1\wedge\ud\bz_1+\kappa_{1\bon\btw}\ud z_1\wedge\ud\bz_2+\kappa_{2\bon\bon}\ud z_2\wedge\ud\bz_1+\kappa_{2\bon\btw}\ud z_2\wedge\ud\bz_2).\end{split}\]

Therefore the holomorphicity of $\theta$ is equivalent to (\ref{loc}) and
\begin{equation}\label{comp}
\begin{split}\frac{\zeta}{2i}(\bpt C\wedge\theta_1-\bpt D\wedge\theta_2)&=C(\kappa_{1\bon\btw}\ud z_1\wedge\ud\bz_1+\kappa_{1\btw\btw}\ud z_1\wedge\ud\bz_2+\kappa_{2\bon\btw}\ud z_2\wedge\ud\bz_1+\kappa_{2\btw\btw}\ud z_2\wedge\ud\bz_2)\\ &-D(\kappa_{1\bon\bon}\ud z_1\wedge\ud\bz_1+\kappa_{1\bon\btw}\ud z_1\wedge\ud\bz_2+\kappa_{2\bon\bon}\ud z_2\wedge\ud\bz_1+\kappa_{2\bon\btw}\ud z_2\wedge\ud\bz_2).\end{split}
\end{equation}
Observe that Equation (\ref{loc}) is (locally) solvable if and only if \[\bpt C\wedge((1-\alpha)\theta_1-2\ud\bz_1)=\bpt D\wedge((1-\alpha)\theta_2+2\ud\bz_2).\] Or in other words, there exists functions $P$, $Q$ and $R$ such that \[\begin{split}\bpt C&=P((1-\alpha)\theta_1-2\ud\bz_1)+Q((1-\alpha)\theta_2+2\ud\bz_2),\\ \bpt D&=-Q((1-\alpha)\theta_1-2\ud\bz_1)+R((1-\alpha)\theta_2+2\ud\bz_2).\end{split}\]

Plug them in Equation (\ref{comp}) we get \[\begin{split}2P\kappa_{1\btw}+2Q\kappa_{1\bon}&=C\kappa_{1\bon\btw}-D\kappa_{1\bon\bon},\\ 2R\kappa_{1\bon}-2Q\kappa_{1\btw}&=C\kappa_{1\btw\btw}-D\kappa_{1\bon\btw},\\ 2P\kappa_{2\btw}+2Q\kappa_{2\bon}&=C\kappa_{2\bon\btw}-D\kappa_{2\bon\bon},\\ 2R\kappa_{2\bon}-2Q\kappa_{2\btw}&=C\kappa_{2\btw\btw}-D\kappa_{2\bon\btw}.\end{split}\]

These equations can be solved explicitly \[\begin{split}P&=2C(\kappa_{1\bon}\kappa_{2\bon\btw}-\kappa_{2\bon}\kappa_{1\bon\btw}) +2D(\kappa_{2\bon}\kappa_{1\bon\bon}-\kappa_{1\bon}\kappa_{2\bon\bon}),\\ Q&=2C(\kappa_{2\btw}\kappa_{1\bon\btw}-\kappa_{1\btw}\kappa_{2\bon\btw}) +2D(\kappa_{1\btw}\kappa_{2\bon\bon}-\kappa_{2\btw}\kappa_{1\bon\bon}),\\ &=2C(\kappa_{2\bon}\kappa_{1\btw\btw}-\kappa_{1\bon}\kappa_{2\btw\btw}) +2D(\kappa_{1\bon}\kappa_{2\bon\btw}-\kappa_{2\bon}\kappa_{1\bon\btw}),\\ R&=2C(\kappa_{2\btw}\kappa_{1\btw\btw}-\kappa_{1\btw}\kappa_{2\btw\btw}) +2D(\kappa_{1\btw}\kappa_{2\bon\btw}-\kappa_{2\btw}\kappa_{1\bon\btw}).\end{split}\]

Substitute $P$, $Q$ and $R$ above back into (\ref{comp}), we get exactly (\ref{simp}).
\end{proof}

\begin{rmk}
The Equation (\ref{simp}) is consistent with itself in the sense that we do not get any further information from taking $\bpt$ of (\ref{simp}).
\end{rmk}

Now we are ready to compute the curvature term $\tr(R\wedge R)$ with respect to the metric (\ref{ansatz}). To do so, we shall let $s=1+|\zeta|^2$ and use $K=(\kappa^{i\bj})$ to denote the inverse matrix of $(\kappa_{i\bj})$. In addition, $\langle\cdot,\cdot\rangle$ is to be understood as the Hermitian metric associated with $\omega$.

Let $L_1\ud\zeta+C_1\theta_1-D_1\theta_2$ and $L_2\ud\zeta+C_2\theta_1-D_2\theta_2$ be locally defined linearly independent holomorphic $(1,0)$-forms. It is not hard to compute that \[\begin{split}\langle\ud\zeta,\ud\zeta\rangle&=\frac{s^2}{2e^{2g}},\\ \langle\theta_1,\theta_1\rangle&=\frac{s^2}{e^{2h+g}}\left(\kappa^{1\bar{1}}+\frac{4}{|\zeta|^2}\kappa_{2\bar{2}}\right) =\frac{s^3}{|\zeta|^2e^{2h+g}}\kappa^{1\bar{1}},\\ \langle\theta_1,\theta_2\rangle&=-\frac{s^3}{|\zeta|^2e^{2h+g}}\kappa^{1\bar{2}}\\ \langle\theta_2,\theta_2\rangle&=\frac{s^3}{|\zeta|^2e^{2h+g}}\kappa^{2\bar{2}}.\end{split}\]
If we consider the local frame \[\{\ud\zeta, (\zeta L_1)\ud\zeta+C_1(\zeta\theta_1)+D_1(-\zeta\theta_2),(\zeta L_2)\ud\zeta+C_2(\zeta\theta_1)+D_2(-\zeta\theta_2)\},\] then the Gram matrix $H$ is given by \[H=\frac{s^2}{2e^{2g}}\begin{pmatrix}1\\ L\end{pmatrix}\begin{pmatrix}1&\bar{L}^T\end{pmatrix}+\frac{s^3}{e^{2h+g}}\begin{pmatrix}0&0\\ 0&EK\bar{E}^T\end{pmatrix},\] where \[L=\begin{pmatrix}\zeta L_1\\ \zeta L_2\end{pmatrix}\] and \[E=\begin{pmatrix}C_1 & D_1\\ C_2 & D_2\end{pmatrix}.\]

Let $U=EK\bar{E}^T$, notice that \[\det H=\frac{s^8}{2e^{4g+4h}}\det U,\] so we can compute that \[H^{-1}=\frac{2e^{2g}}{s^2}\begin{pmatrix}1&0&0\\0&0&0\\ 0&0&0\end{pmatrix}+\frac{e^{2h+g}}{s^3}\begin{pmatrix}-\bar{L}^T\\ I_2\end{pmatrix}U^{-1}\begin{pmatrix}-L&I_2\end{pmatrix},\] where $I_2$ is the $2\times 2$ identity matrix. Write \[A=\frac{s^2}{2e^{2g}}~\textrm{ and }~B=\frac{s^3}{e^{2h+g}}.\] It follows that \[\begin{split}\pt\bar{H}&=\pt A\begin{pmatrix}1\\ \bar{L}\end{pmatrix}\begin{pmatrix}1& L^T\end{pmatrix}+A\begin{pmatrix}0\\ \pt\bar{L}\end{pmatrix}\begin{pmatrix}1& L^T\end{pmatrix}+A\begin{pmatrix}1\\ \bar{L}\end{pmatrix}\begin{pmatrix}0& \pt L^T\end{pmatrix}+\pt B\begin{pmatrix}0&0\\ 0&\bar{U}\end{pmatrix}+B\begin{pmatrix}0&0\\ 0&\pt\bar{U}\end{pmatrix},\\ \bar{H}^{-1}\pt\bar{H}&=\pt\log A\begin{pmatrix}1&L^T\\ 0&0\end{pmatrix}+\begin{pmatrix}0&\pt L^T\\ 0&0\end{pmatrix}+\frac{A}{B}\begin{pmatrix}-L^T\\ I_2\end{pmatrix}\bar{U}^{-1}\pt\bar{L}\begin{pmatrix}1&L^T\end{pmatrix}+\pt\log B\begin{pmatrix}0& -L^T\\ 0& I_2\end{pmatrix}\\ &+\begin{pmatrix}-L^T\\ I_2\end{pmatrix}\begin{pmatrix}0& \bar{U}^{-1}\pt\bar{U}\end{pmatrix}.\end{split}\]

As a consequence, \[\begin{split}R&=\bpt(\bar{H}^{-1}\pt\bar{H})\\ &=\bpt\pt\log A\begin{pmatrix}1&L^T\\ 0&0\end{pmatrix}-\pt\log A\begin{pmatrix}0&\bpt L^T\\ 0&0\end{pmatrix}+\begin{pmatrix}0&\bpt\pt L^T\\ 0&0\end{pmatrix}+\bpt\left(\frac{A}{B}\right)\begin{pmatrix}-L^T\\ I_2\end{pmatrix}\bar{U}^{-1}\pt\bar{L}\begin{pmatrix}1&L^T\end{pmatrix}\\ &-\frac{A}{B}\begin{pmatrix}\bpt L^T\\ 0\end{pmatrix}\bar{U}^{-1}\pt\bar{L}\begin{pmatrix}1&L^T\end{pmatrix}+\frac{A}{B}\begin{pmatrix}-L^T\\ I_2\end{pmatrix}\bpt(\bar{U}^{-1}\pt\bar{L})\begin{pmatrix}1&L^T\end{pmatrix}-\frac{A}{B}\begin{pmatrix}-L^T\\ I_2\end{pmatrix}\bar{U}^{-1}\pt\bar{L}\begin{pmatrix}0&\bpt L^T\end{pmatrix}\\ &+\bpt\pt\log B\begin{pmatrix}0& -L^T\\ 0& I_2\end{pmatrix}+\pt\log B\begin{pmatrix}0& \bpt L^T\\ 0& 0\end{pmatrix}-\begin{pmatrix}\bpt L^T\\ 0\end{pmatrix}\begin{pmatrix}0& \bar{U}^{-1}\pt\bar{U}\end{pmatrix}+\begin{pmatrix}-L^T\\ I_2\end{pmatrix}\begin{pmatrix}0& \bpt(\bar{U}^{-1}\pt\bar{U})\end{pmatrix}.\end{split}\]

A lengthy calculation shows that
\begin{equation}\label{c1}
\tr(R)=\bpt\pt\log A+2\bpt\pt\log B+\tr(\bpt(\bar{U}^{-1}\pt\bar{U}))
\end{equation}
and
\begin{equation}\label{c2comp}
\tr(R\wedge R)=2\pt\bpt\left(\frac{A}{B}W\right)+2(\bpt\pt\log B)^2+2\bpt\pt\log B\wedge\tr(\bpt(\bar{U}^{-1}\pt\bar{U}))+\tr(\bpt(\bar{U}^{-1}\pt\bar{U})^2),
\end{equation}
where $W=\bpt L^T\bar{U}^{-1}\pt\bar{L}$ and we have used the fact that $(\bpt\pt\log A)^2=0$.

Recall from (\ref{mod}) that if $\|\Omega\|_\omega$ is a constant, or equivalently, $e^{h+g}=s^2$, then $\tr(R)=0$. From this and (\ref{c1}) we can deduce that
\begin{equation}\label{trace}
\tr(\bpt(\bar{U}^{-1}\pt\bar{U}))=0.
\end{equation}

From the (restricted) twistor fibration $p:Y\to\ccc$, we have a short exact sequence of holomorphic vector bundles \[0\to p^*T^*\ccc\to T^*Y\xrightarrow{q}E'=T^*Y/p^*T^*\ccc\to0.\] Moreover, the matrix $U$ is the Gram matrix of the holomorphic frame of $E'$ induced by \[\{(\zeta L_1)\ud\zeta+C_1(\zeta\theta_1)+D_1(-\zeta\theta_2),(\zeta L_2)\ud\zeta+C_2(\zeta\theta_1)+D_2(-\zeta\theta_2)\}\] with respect to the natural metric scaled by $s$. Let $F'$ be the curvature form of $E'$ with respect to this metric, then we have
\begin{equation}\label{c2'}
F'=\bpt(\bar{U}^{-1}\pt\bar{U}).
\end{equation}

By making use of (\ref{trace}) and (\ref{c2'}), we can simplify (\ref{c2comp}) to
\begin{equation}
\label{c2simp}\tr(R\wedge R)=2\pt\bpt\left(\frac{A}{B}W\right)+2(\bpt\pt\log B)^2+\tr(F'\wedge F').
\end{equation}

The quantity $W$ can be computed explicitly using (\ref{loc}). Notice that \[2\bpt L^T=\begin{pmatrix}2\ud\bz_1-(1-\alpha)\theta_1 & 2\ud\bz_2+(1-\alpha)\theta_2\end{pmatrix}E^T,\] it follows that \[\begin{split}W&=\frac{1}{4}\begin{pmatrix}2\ud\bz_1-(1-\alpha)\theta_1 & 2\ud\bz_2+(1-\alpha)\theta_2\end{pmatrix}\bar{K}^{-1}\begin{pmatrix}2\ud z_1-(1-\alpha)\bar{\theta}_1 \\ 2\ud z_2+(1-\alpha)\bar{\theta}_2\end{pmatrix}\\ &=\frac{2i}{s}(\alpha\omega_I+\beta\omega_J+\gamma\omega_K).\end{split}\]

Consequently, the anomaly cancellation equation (\ref{ac}) reduces to
\begin{equation}\label{acsimp}
\begin{split}
&i\pt\bpt\left(\frac{e^{2h+g}}{s^2}(\alpha\omega_I+\beta\omega_J+\gamma\omega_K)\right)\\ =& \frac{\alpha'}{4}\left[2i\pt\bpt\left(\frac{e^{2h-g}}{s^2}(\alpha\omega_I+\beta\omega_J+\gamma\omega_K)\right)+2(\bpt\pt\log B)^2+\tr(F'\wedge F')-\tr(F\wedge F)\right]
\end{split}
\end{equation}
and we are free to choose functions $g$ and $h$.

The simplest way to let (\ref{acsimp}) hold is to take $g=\dfrac{1}{2}\log\dfrac{\alpha'}{2}$, $F=F'$ and choose appropriate $h$ such that
\begin{equation}\label{sq}
(\bpt\pt\log B)^2=0.
\end{equation}
Recall that $B=s^3/e^{2h+g}$, hence (\ref{sq}) holds trivially if $h$ is a constant say $h\equiv0$. In this case, the metric (\ref{ansatz}) is conformal to the product of hyperk\"ahler metric on $N$ and the Euclidean metric on $\ccc$, hence conformally Ricci-flat. It should be pointed out that this metric is not complete, however this does not raise any problem for the use of gluing. Intuitively we can think of $\omega$ as a metric on the singular space $\ccc\pp^1\times N/\{\infty\}\times N$.

It is also possible that (\ref{sq}) holds for nonconstant $h$. To find such $h$, one usually needs to know the explicit hyperk\"ahler metric on $N$. We will give an example of this kind in Section \ref{s4}.

\subsection{The Hermitian-Yang-Mills Equation (\ref{hym})}

In \cite{fei2015c}, $N$ is taken to be flat $T^4$ hence $F'=F=0$ and (\ref{hym}) holds automatically. In general, neither $N$ is flat nor $F'$ trivial. In order to solve the whole Strominger, we shall prove
\begin{thm}\label{thym}
$F'$ solves the Hermitian-Yang-Mills equation, i.e., \[F'\wedge\omega^2=0\] for arbitrary $g$ and $h$.
\end{thm}
By the product structure $Y=\ccc\times N$, we have the decomposition of space of 2-forms \[\Omega^2(Y)=\Omega^2(\ccc)\oplus\Omega^1(\ccc)\otimes\Omega^1(N)\oplus\Omega^2(N).\] Moreover the complex structure $\frI$ is compatible with this decomposition, so we have \[\Omega^{1,1}(Y)=\Omega^{1,1}(\ccc)\oplus\Omega^{1,0}(\ccc)\otimes\Omega^{0,1}(N)\oplus\Omega^{0,1}(\ccc)\otimes\Omega^{1,0}(N) \oplus\Omega^{1,1}(N).\]
We first prove the following lemma
\begin{lemma}
Every entry of $F'$ is contained in the space \[\Omega^{1,0}(\ccc)\otimes\Omega^{0,1}(N)\oplus\Omega^{0,1}(\ccc)\otimes\Omega^{1,0}(N) \oplus\Omega^{1,1}(N)\].
\end{lemma}
\begin{proof}
Recall that $U=EK\bar{E}^T$, so we have \[F'=\bpt(\bar{U}^{-1}\pt\bar{U})=\bpt((E^T)^{-1}\bar{K}^{-1}\bar{E}^{-1}\pt\bar{E}\cdot\bar{K}E^T)+\bpt( (E^T)^{-1}\bar{K}^{-1}\pt\bar{K}\cdot E^T)+\bpt((E^T)^{-1}\pt E^T)).\] Observe that (\ref{simp}) can be rewritten as \[\bpt E=i(\beta-i\gamma)E\cdot\begin{pmatrix}\kappa_{1\bon\btw}\overline{\theta}_1-\kappa_{2\bon\btw}\overline{\theta}_2 & \kappa_{1\btw\btw}\overline{\theta}_1-\kappa_{2\btw\btw}\overline{\theta}_2\\ \kappa_{2\bon\bon}\overline{\theta}_2-\kappa_{1\bon\bon}\overline{\theta}_1 & \kappa_{2\bon\btw}\overline{\theta}_2-\kappa_{1\bon\btw}\overline{\theta}_1\end{pmatrix},\] which does not have any component from $\Omega^1(\ccc)$. Similarly $\pt\bar{K}$ contains only components in $\Omega^{1,0}(N)$. The lemma follows directly from these two observations.
\end{proof}
Now we proceed to prove Theorem \ref{thym}
\begin{proof}
From (\ref{bal}), we see that $\omega^2$ lives in the space \[\Omega^2(\ccc)\otimes\Omega^2(N)\oplus\Omega^4(N).\] Therefore the only component of $F'$ that would contribute in $F'\wedge\omega^2$ is its $\Omega^{1,1}(N)$-part. The key point is that the $\Omega^{1,1}(N)$-part of $F'$ can be computed fiberwise. Fix $\zeta\in\ccc$ and let $N_\zeta$ be the fiber of the twistor fibration over $\zeta$. Notice that $(N_\zeta,\omega_\zeta:=\alpha\omega_I+\beta\omega_J+\gamma\omega_K)$ is hyperk\"ahler. Moreover, $E'|_{N_{\zeta}}$ is the cotangent bundle of $N_\zeta$ with hyperk\"ahler metric. It is a well-known fact that hyperk\"ahler 4-manifolds are anti-self-dual, thus \[F'_{\Omega^{1,1}(N)}\wedge\omega_\zeta=0.\] The theorem follows directly.
\end{proof}

In summary, for any hyperk\"ahler 4-manifold $N$, we can construct a noncompact 3-fold $Y$ with trivial canonical bundle, which can be obtained by removing an arbitrary fiber of the twistor fibration $p:Z\to\ccc\pp^1$ from $Z$. It is worth mentioning that the structure of $Y$ depends on the choice of fiber. The main theorem we have proved is the following.
\begin{thm}
The whole Strominger system (\ref{hym}, \ref{ac}, \ref{cb}) can be solved on $Y$ for arbitrary $N$ hyperk\"ahler.
\end{thm}

\section{Some Examples}\label{s4}

In this section, we will investigate the cases $N=\rr^4$ and $N=T^*\ccc\pp^1$ in detail.

\subsection{$N=\rr^4$}

Identify $\rr^4$ with the space of quaternions $\hh$, then left multiplication by $i$, $j$ and $k$ defines the standard hyperk\"ahler structure on $\rr^4$. We can construct the space $Y$ as in Section \ref{s3} by removing the fiber with complex structure $-I$ at infinity. Actually, $Y$ is biholomorphic to $\ccc^3$. An explicit isomorphism  $f:\ccc^3=\ccc\times\ccc^2\to Y=\ccc\times\rr^4$ can be written down as \[(\zeta,u_1,u_2)=f(\zeta,w_1,w_2)=(\zeta,\frac{w_1-i\zeta\bw_2}{1+|\zeta|^2},\frac{w_2+i\zeta\bw_1}{1+|\zeta|^2}).\] Or conversely, \[\begin{cases}w_1&=u_1+i\zeta\bu_2\\ w_2&=u_2-i\zeta\bu_1\end{cases},\] where $u_1=x^1+ix^2$, $u_2=x^3+ix^4$ and $u^1+u^2j=x^1+x^2i+x^3j+x^4k$ parameterizes $\hh=\rr^4$. In terms of the holomorphic coordinate $\{\zeta,w_1,w_2\}$ on $\ccc^3$, the 2-form $\alpha\omega_I+\beta\omega_J+\gamma\omega_K$ can be expressed as
\begin{equation}\label{exp}
\begin{split}\alpha\omega_I+\beta\omega_J+\gamma\omega_K&=\frac{i}{2s}(\ud w_1\wedge\ud\bw_1+\ud w_2\wedge\ud\bw_2+(|u_1|^2+|u_2|^2)\ud\zeta\wedge\ud\bzeta\\ &+iu_2\ud w_1\wedge\ud\bzeta-iu_1\ud w_2\wedge\ud\bzeta-i\bu_2\ud\zeta\wedge\ud\bw_1+i\bu_1\ud\zeta\wedge\ud\bw_2).\end{split}
\end{equation}

Now let us consider Equation (\ref{sq}) \[(\bpt\pt\log B)^2=0.\] Notice that \[\bpt\pt\log B=(3\bpt\pt\log s-\bpt\pt g)-2\bpt\pt h,\] hence (\ref{sq}) is equivalent to that
\begin{equation}\label{sqs}
(\bpt\pt h)^2=\bpt\pt h\wedge(3\bpt\pt\log s-\bpt\pt g).
\end{equation}

For simplicity, let us assume that $h:\rr^4\to\rr$ is a radial function, i.e., $h=h(\rho)$ where \[\rho=|u_1|^2+|u_2|^2=(x^1)^2+(x^2)^2+(x^3)^2+(x^4)^2.\] Therefore \[\ud h=h'\ud\rho\] and \[\begin{split}2i\bpt\pt h&=\ud\frI\ud h\\ &=h''\ud\rho\wedge(\alpha I+\beta J+\gamma K)\ud\rho+h'(\ud\alpha\wedge I\ud\rho+\ud\beta\wedge J\ud\rho+\ud\gamma\wedge K\ud\rho)-4h'(\alpha\omega_I+\beta\omega_J+\gamma\omega_K).\end{split}\]
One can verify that
\[\begin{split}-4(\bpt\pt h)^2&=32h'(h'+h''\rho)\vol_{\rr^4}-2(h')^2(\alpha J\ud\rho\wedge K\ud\rho+\beta K\ud\rho\wedge I\ud\rho+\gamma I\ud\rho\wedge J\ud\rho)\wedge\omega_{\ccc\pp_1}\\ &+8h'(h'+h''\rho)((\alpha\ud\beta-\beta\ud\alpha)\wedge*K\ud\rho+(\beta\ud\gamma-\gamma\ud\beta)\wedge*I\ud\rho+(\gamma\ud\alpha -\alpha\ud\gamma)\wedge*J\ud\rho),\end{split}\] where $*$ is the Hodge star operator on $\rr^4$. Also we have
\[4\bpt\pt h\wedge(3\bpt\pt\log s-\bpt\pt g)=(h''\ud\rho\wedge(\alpha I+\beta J+\gamma K)\ud\rho-4h'(\alpha\omega_I+\beta\omega_J+\gamma\omega_K))\wedge(2i\bpt\pt\log g+3\omega_{\ccc\pp^1}).\]
Notice that we have the identity \[\ud\rho\wedge I\ud\rho+4\rho\cdot\omega_I=J\ud\rho\wedge K\ud\rho\] and the like. It follows that assuming $g$ is a constant, then (\ref{sqs}) holds if and only if $h'=0$ or $h'=\dfrac{-3}{2\rho}$. In both cases we solve the Strominger system.
\begin{enumerate}
\item The case $h'=0$.

The metric on $Y$ is essentially of the form \[\omega=\frac{1}{s^2}\left(\alpha\omega_I+\beta\omega_J+\gamma\omega_K+\frac{i}{2}\ud\zeta\wedge\ud\bzeta\right),\] which is conformal to the Euclidean metric on $Y=\ccc\times\rr^4$. However this metric has rather complicated expression (see (\ref{exp})) in standard complex coordinates on $\ccc^3$. Direct calculation shows that though its sectional curvature is not bounded from below, however it is bounded from above by a positive constant. Same conclusion applies to Ricci and scalar curvatures as well.

\item The case $h'=\dfrac{-3}{2\rho}$.

In this case $e^h\sim\rho^{-3/2}$, so the metric $\omega$ is only defined on $\ccc\times(\rr^4\setminus\{0\})\cong\ccc\times(\ccc^2\setminus\{0\})$. On each copy of $\rr^4\setminus\{0\}$, the restricted metric is conformally flat, with non-positive sectional curvature. The curvature property of $\ccc\times(\rr^4\setminus\{0\})$ is like what we have in the $h'=0$ case.
\end{enumerate}

\subsection{$N=T^*\ccc\pp^1$}

The hyperk\"ahler structure on $T^*\ccc\pp^1$ is known as the Eguchi-Hanson \cite{eguchi1978} geometry, which was also discovered by Calabi \cite{calabi1979} independently. This is the simplest example of ALE space which can be constructed by taking resolution of $\ccc^2/\zz_2$. Now let $N$ be the Eguchi-Hanson space. As we have seen from Proposition \ref{resolved} that if we choose the fiber removed to be special, then $Y$ is can be identified with the resolved conifold.

As we know, setting $h\equiv\textrm{const}$ gives a solution to the Strominger system which is conformal to the product metric. Its curvature behaviors are very similar to the corresponding $N=\rr^4$ case.

One may want to ask if there is a nonconstant $h$ such that (\ref{sq}), or equivalently (\ref{sqs}) holds. We may want to make a reduction as before by assuming that $h$ depends only on $R$, where $R:T^*\ccc\pp^1\to\rr$ is the norm square function of cotangent fibers with respect to the Fubini-Study metric $\omega_{\ccc\pp^1}$. Unfortunately in such a case we do not get any extra solutions by this consideration.

\section{On Chern-Ricci Flat Balanced Metrics}\label{s5}

Let $X$ be a K\"ahler manifold with trivial canonical bundle. People are very interested in the question that if $M$ admits a (complete) Ricci-flat K\"ahler metric or not. For $X$ compact, this was answered affirmatively by Yau's famous solution to the Calabi conjecture \cite{yau1977, yau1978}. However, we do not know enough about this problem when $X$ is noncompact. For instance, let $M$ be any K\"ahler manifold, then we know that the total space of the canonical bundle of $M$, denoted by $K_M$, is K\"ahler and has itself trivial canonical bundle. So it makes sense to ask when does $K_M$ support a K\"ahler Ricci-flat metric. This problem was first attacked by Calabi \cite{calabi1979}, where he showed that if $M$ admits a K\"ahler-Einstein metric, then one can write down a K\"ahler Ricci-flat metric on $K_M$. A relative recent progress in this direction was made by Futaki \cite{futaki2007}, where he showed that such a metric exists if $M$ is toric Fano.

The Calabi ansatz can be rephrased as follows. Let $(M,\omega)$ be a K\"ahler manifold. By choosing a set of holomorphic coordinates $\{z^1,\dots,z^n\}$ on $M$, we can trivialize $K_M$ by $\ud z^1\wedge\dots\wedge\ud z^n$ locally. Let $t$ parameterizes the fiber of $K_M$ under this trivialization, then $\{z^1,\dots,z^n,t\}$ forms a set of holomorphic coordinates on $K_M$ and \[\Omega=\ud z^1\wedge\dots\wedge\ud z^n\wedge\ud t\] is a globally defined nowhere vanishing holomorphic volume form.

In terms of coordinate $\{z^1,\dots,z^n\}$, the K\"ahler form $\omega$ can be written as \[\omega=ih_{j\bk}\ud z^j\wedge\ud\bz^k.\] It follows that  $h=\det(h_{j\bk})$ is a positive function. Notice that the K\"ahler metric on $M$ naturally induces an Hermitian metric on $K_M$, which can be expressed as \[\omega_0=\omega+\frac{i}{h}(\ud t-t\pt\log h)\wedge(\ud\bar{t}-\bar{t}\bpt\log h).\] Let $R:K_M\to\rr$ be the norm square function of fibers of $K_M\to M$. Clearly $R=\dfrac{|t|^2}{h}$ and the metric $\omega_0$ has the form \[\omega_0=\omega+i\frac{\pt R\wedge\bpt R}{R}.\] In general, $\omega_0$ is not a K\"ahler metric. It turns out that $\omega_0$ is K\"ahler if and only if \[\pt\bpt\log R=-\pt\bpt\log h=0,\] i.e., $(M,\omega)$ is K\"ahler Ricci-flat.

To get a better behaved metric, we can modify $\omega_0$ by some conformal factors. Let $u,v$ be smooth functions on $M$, and $f,g$ be smooth functions of $R$, then one can cook up a new Hermitian metric \[\omega_{u,v,f,g}=e^{u+f}\omega+ie^{v+g}\frac{\pt R\wedge\bpt R}{R}.\] It is not hard to check that $\Omega$ is of constant length if and only if
\begin{equation}\label{chern-ricci}
v=-nu\quad\textrm{and}\quad g=-nf+c
\end{equation}
for some constant $c$. Assuming this, if we further want the metric $\omega_{u,v,f,g}$ to be K\"ahler, then $u$ must be a constant. Without loss of generality we may assume that $u=0$, and we still need \[e^f\pt f\wedge\omega-ie^{-nf+c}\pt R\wedge\pt\bpt\log R=0.\] In other words, \[e^{(n+1)f-c}f'\omega=i\pt\bpt\log R=-i\pt\bpt\log h=Ric(\omega).\] We see immediately that this equation has a solution if and only if $\omega$ is K\"ahler-Einstein, in which case we get the Calabi ansatz.

In this section, we shall consider the case that $\omega_{u,v,f,g}$ is balanced and Chern-Ricci flat. Clearly such a metric satisfy the conformally balanced equation (\ref{cb}), which geometrically means that the manifold has restricted holonomy contained in $SU(n)$ with respect to the Strominger-Bismut connection. The condition we are considering far more stronger, which can be interpreted as that $\Omega$ is parallel under the Strominger-Bismut connection.

Now let us impose the balanced Chern-Ricci flat condition on $\omega_{u,v,f,g}$. For the ``Chern-Ricci flat'' part, i.e., $\|\Omega\|_{\omega_{u,v,f,g}}\equiv\textrm{const.}$, we still need (\ref{chern-ricci}) as before. Plug this in and we can compute that \[\omega_{u,v,f,g}^n=e^{n(u+f)}\omega^n+ine^{c-u-f}\omega^{n-1}\wedge\frac{\pt R\wedge\bpt R}{R}.\] Hence \[\pt(\omega_{u,v,f,g}^n)=e^{n(u+f)}\pt f\wedge\omega^n-ine^{c-u-f}\omega^{n-1}\wedge\pt R\wedge\pt\bpt\log R-ine^{c-u-f}\omega^{n-1}\wedge\pt u\wedge\frac{\pt R\wedge\bpt R}{R}.\] As there is no other terms to cancel the last term in the above equation, so we need $\pt u=0$ to make $\omega_{u,v,f,g}$ balanced. By choosing $u=0$, the balancing codition is reduced to \[e^{nf}\pt f\wedge\omega^n=ine^{c-f}\omega^{n-1}\wedge\pt R\wedge\pt\bpt\log R,\] or equivalently, \[e^{(n+1)f-c}f'\omega^n=-in\omega^{n-1}\pt\bpt\log h=s\cdot\omega^n,\] where $s$ is the scalar curvature function of $M$ up to a positive constant. From the calculation we conclude that this is possible if and only is $(M,\omega)$ has constant scalar curvature.

Constant scalar curvature K\"ahler metrics (cscK) has been studied extensively as a special case of extremal K\"ahler metrics \cite{calabi1982}. It is believed that the existence of cscK metrics is equivalent to certain stability condition in the sense of algebraic geometry.

Notice that in the derivation of Chern-Ricci flat balanced metric on $K_M$, what we actually need is that $\omega$ is balanced instead of being K\"ahler. In such a case, $s$ is known as the Chern scalar curvature, which is in general different from the scalar curvature in Riemannian geometry. Thus we have proved the following generalization of Calabi's result.
\begin{thm}\label{chernricci}
If $M$ admits a balanced metric with constant Chern scalar curvature, then $K_M$ admits a Chern-Ricci flat balanced metric.
\end{thm}

Recall that on a complex $n$-fold $X$, a balanced metric $\omega$ defines the so-called balanced class 
\[\left[\frac{\omega^{n-1}}{(n-1)!}\right]\in H^{n-1,n-1}_{BC}(X)=\dfrac{\ud\textrm{-closed }(n-1,n-1)\textrm{-forms}}{i\pt\bpt\textrm{-exact }(n-1,n-1)\textrm{-forms}}.\]
The balanced version of Gauduchon conjecture has been proposed for some time, see \cite[Conjecture 4.1]{tosatti2015} for instance. In particular, for balanced manifolds with trivial canonical bundle, this conjecture implies the existence of Chern-Ricci flat balanced metrics in any given balanced class. We refer to \cite{szekelyhidi2015} for recent progresses on this problem.

Theorem \ref{chernricci} implies that Chern-Ricci flat balanced metrics should be viewed as a balanced analogue of extremal K\"ahler metrics. We shall justify this claim by the following consideration.

On a compact balanced manifold $(X,\omega)$, a very useful property is that the total Chern scalar curvature
\[\int_Xs\cdot\frac{\omega^n}{n!}=\int_X\rho\wedge\frac{\omega^{n-1}}{(n-1)!}\]
depends only on the complex structure of $X$ and the balanced class. Here $\rho=-i\pt\bpt\log h$ is the Chern-Ricci form. As an analogue of K\"ahler csae, it is natural to consider the variational problem associated to the Calabi-type functional \cite{calabi1982}
\[\cls(\omega)=\int_Xs^2\cdot\frac{\omega^n}{n!},\]
where $\dfrac{\omega^{n-1}}{(n-1)!}$ is allowed to vary in a given balanced class. Since we can modify $\omega$ by $i\pt\bpt\alpha$, where $\alpha$ is any $(n-2,n-2)$-form, we expect the associated Euler-Lagrange equation is an equation of $(n-2,n-2)$-forms, or dually, a $(2,2)$-form equation. 

Indeed, it is not hard to derive the following
\begin{thm}
A balanced metric $\omega$ is a critical point of $\cls(\omega)$ if and only if it satisfies
\begin{equation}\label{extremalbalance}
2(n-1)i\pt\bpt s\wedge\rho=i\pt\bpt((2\Delta s+s^2)\omega),
\end{equation}
where $\Delta_c$ is the complex Laplacian defined by
\[\Delta f=\Lambda(i\pt\bpt f)=h^{k\bj}\frac{\pt^2f}{\pt z^j\pt\bz^k}.\]
We shall call such balanced metrics extremal.
\end{thm}
From Equation (\ref{extremalbalance}), we have the following observations:
\begin{enumerate}
\item If the background metric is non-K\"ahler, one can easily show that $i\pt\bpt\omega\neq0$. Hence $s=\textrm{const.}$ is an extremal balanced metric if and only if $s=0$. In particular, Chern-Ricci flat balanced metrics are extremal. An intuitive reason is that there is no analogue of K\"ahler-Einstein metrics with nonzero Einstein constant on non-K\"ahler balanced manifolds. Indeed, if there is a smooth function $f$ such that \[\rho=f\cdot\omega,\] one can deduce that $\rho\equiv0$.
\item If there exists a (1,0)-form $\mu$ such that
    \begin{equation}\label{aepplibalanced}
    2(n-1)s\cdot\rho=(2\Delta s+s^2)\omega+\bpt\mu+\pt\bar{\mu},
    \end{equation}
    then (\ref{extremalbalance}) holds. Such a condition is automatically satisfied if $H^{1,1}_A(X)=0$, where $H^*_A(X)$ is the Aeppli cohomology group of $X$. In this case, by taking trace of (\ref{aepplibalanced}), we get
    \[2(n-1)s^2=2n\Delta s+ns^2+\Lambda(\bpt\mu+\pt\bar{\mu}).\]
    Since the last term is of divergence form, by integration over $X$, we get
    \[(n-2)\int_Xs^2\cdot\frac{\omega^n}{n!}=0.\]
    As we always assume that $n>2$ (otherwise $\omega$ is K\"ahler), we conclude that $s\equiv0$.
\item Assuming $s\equiv0$, if we further assume that $0=c_1(X)\in H^{1,1}_{BC}(X)$, then $\rho=i\pt\bpt f$ for some globally defined real function $f$. Therefore
    \[0=s=\Lambda\rho=\Delta f.\]
    Hence by maximal principle $f$ is a constant and $\rho\equiv0$.
\end{enumerate}
A very important class of non-K\"ahler Calabi-Yau 3-folds is of the form $X_k=\#_k(S^3\times S^3)$ for $k\geq2$ \cite{friedman1991, lu1996}, which can be constructed from projective Calabi-Yau 3-folds by taking conifold transitions. Moreover, these manifolds satisfy the $\pt\bpt$-lemma and admit balanced metrics \cite{fu2012}. Therefore conditions in (b) and (c) above are satisfied. Consequently we have
\begin{cor}
Extremal balanced metrics on $X_k$ are exactly those Chern-Ricci flat balanced metrics.
\end{cor}
Hopefully this point of view will be useful in proving the balanced Gauduchon conjecture for $X_k$'s.
 
\bibliographystyle{alpha}

\bibliography{C:/Users/Piojo/Dropbox/Documents/Source}

\end{document}